\documentclass[10pt]{amsart}%
\usepackage{amsmath}
\usepackage{amsfonts}
\usepackage{amsthm}
\usepackage{amssymb}
\usepackage{graphicx}%
\newtheorem{theorem}{Theorem}
\newtheorem*{theorem*}{Theorem}
\newtheorem*{acknowledgement*}{Acknowledgement}
\newtheorem*{definition*}{Definition}

\newtheorem{lemma}[theorem]{Lemma}

\newcommand{\pdt}[0]{\frac{\partial}{\partial t}}

\newcommand{\gt}[0]{\tilde{g}}

\newcommand{\nabh}[0]{\hat{\nabla}}

\newcommand{\Rc}[0]{\operatorname{Rc}}
\newcommand{\Rm}[0]{\operatorname{Rm}}

\newcommand{\Xt}[0]{\tilde{X}}
\newcommand{\Yt}[0]{\tilde{Y}}
\newcommand{\Rmt}[0]{\widetilde{\operatorname{Rm}}}

\newcommand{\dfn}[0]{\doteqdot}

\newcommand{\wt}[1]{\widetilde{#1}}

\newcommand{\nabt}[0]{\wt{\nabla}}

\newcommand{\pdtau}[0]{\partial_{\tau}}
\newcommand{\Deltat}[0]{\widetilde{\Delta}}

\newcommand{\Ec}[0]{\mathcal{E}}
\newcommand{\Fc}[0]{\mathcal{F}}

\newcommand{\Ic}[0]{\mathcal{I}}

\newcommand{\Lc}[0]{\mathcal{L}}
\newcommand{\Nc}[0]{\mathcal{N}}
\newcommand{\Wc}[0]{\mathcal{W}}
\newcommand{\Xc}[0]{\mathcal{X}}
\newcommand{\Yc}[0]{\mathcal{Y}}

\newcommand{\Vt}[0]{\tilde{V}}
\newcommand{\Et}[0]{\tilde{E}}

\newcommand{\supp}[0]{\operatorname{supp}}

\title[Backward uniqueness for some geometric evolution equations]
{A short proof of backward uniqueness for some geometric evolution equations.}

\author{Brett Kotschwar}
\address{Arizona State University, Tempe, AZ, USA}
\email{kotschwar@asu.edu}

\thanks{The author was supported in part by NSF grant DMS-1160613.}

\date{Aug 2013}

\begin{document}
\begin{abstract}
We give a simple, direct proof of the backward uniqueness
of solutions to a class of second-order geometric evolution equations including 
the Ricci and cross-curvature flows.
The proof, based on a classical argument of Agmon-Nirenberg,
uses the logarithmic convexity of a certain energy quantity in the place of Carleman inequalities. 
We also demonstrate the applicability of the technique to the $L^2$-curvature flow and
other higher-order equations.
\end{abstract}
\maketitle


\maketitle

\section{Introduction}
The invariance of a curvature flow $\pdt g = E(g)$ on a smooth manifold $M$ 
under the action of $\operatorname{Diff}(M)$ has the important consequence that the isometry group of a solution
remains unchanged within any class
in which the uniqueness of solutions forward and backward in time is assured. 
An operator $E$ with such an invariance, however, 
is never strongly elliptic, and so these same questions of uniqueness are not merely 
consequences of standard parabolic theory.  

Whereas the problem of forward uniqueness for curvature flows can
often be converted to one for a strictly parabolic system by the means of DeTurck's trick (see, e.g., \cite{HamiltonSingularities}, 
\cite{ChenZhu},
\cite{BahuaudHelliwellU}), the ill-posedness of terminal-value problems for parabolic equations 
prohibits the application of this same technique
to the problem of backward uniqueness. This is discussed in some detail
in \cite{KotschwarRFBU} in the context of the Ricci flow. In that paper we show that, nevertheless,
the problem can be recast as one for a prolonged system to which the usual method
of weighted $L^2$ (Carleman-type) estimates can be extended.

 The Carleman inequality approach to problems of backward uniqueness, which originates in the classical work of Lees and Protter \cite{LeesProtter},
\cite{Protter}, is extremely
robust and is the basis of a vast literature of results on the unique continuation of solutions to parabolic equations;
for a picture of the current state-of-the-art, we refer the reader to \cite{Escauriaza}, 
\cite{EscauriazaFernandez}, \cite{EKPV}, \cite{ESS}, \cite{EV}, \cite{Fernandez}, \cite{KochTataru}, \cite{Sogge}, \cite{Tataru},
and the references therein.  For the \emph{global} problem of backward uniqueness, however,
the simplest approach is perhaps still the 
energy-quotient/logarithmic-convexity technique of Agmon and Nirenberg \cite{AgmonNirenberg1}, \cite{AgmonNirenberg2}.
This approach dates to a few years after \cite{LeesProtter} and has since been clarified and 
extended by a number of authors (see, e.g., \cite{BardosTartar}, \cite{Ghidaglia},
\cite{Kukavica1}, \cite{Kukavica2}, \cite{Ogawa}). The analysis of the basic energy quotient on which it is centered
is related in turn to that of the parabolic frequency quantities of Poon \cite{Poon} and others (e.g., \cite{EKPV}).

The primary purpose of this paper is to show that this energy quotient technique can also be applied
to backward uniqueness problems for geometric flows once they have been appropriately reformulated.
In Section \ref{sec:2ndorderbu}, we use this technique to give a shorter and more transparently 
quantitative proof of the general backward uniqueness theorem
for systems of mixed differential inequalities in \cite{KotschwarRFBU}. In Section \ref{sec:2ndorderapps}, we discuss 
the application of this result to the problem of backward uniqueness for the Ricci and cross-curvature flows. 
In Section \ref{sec:l2}, we use this technique to prove a backward uniqueness theorem for solutions
to the fourth-order $L^2$-curvature flow and discuss extensions to a class of of higher-order equations.

We have already mentioned the connection of uniqueness results for geometric flows
to the preservation of the isometry group of a solution: for such flows, a backward uniqueness theorem implies that solutions do
not acquire new isometries during their smooth lifetimes. There are a number of 
other qualitative properties of geometric flows, however, which, when suitably formulated,
 amount to statements of backward uniqueness for some system associated to the flow. Among them,
for example, are the impossibility of a classical solution to the Ricci flow
becoming spontaneously Einstein or self-similar  at some point in its evolution \cite{KotschwarRFBU}
 or acquiring a novel K\"ahler or local product structure \cite{KotschwarRFHolonomy}. A backward uniqueness problem (albeit a somewhat nonstandard one -- see \cite{ESS})
is also at the heart of our our recent work \cite{KotschwarWang} with Lu Wang on the asymptotic rigidity of shrinking Ricci solitons. The 
connection between this particular type of problem of elliptic unique-continuation-at-infinity and 
backward uniqueness was first illustrated by Wang in \cite{WangMCFCone} (cf. \cite{WangMCFCyl}) in the context of the mean curvature flow.

\section{A general backward uniqueness theorem for second-order systems}
\label{sec:2ndorderbu}
\subsection{Definitions and conventions}
In this section, $M = M^n$ will denote a connected smooth manifold equipped with a smooth family $g(\tau)$
of Riemannian metrics defined for $\tau\in [0, \Omega]$; the Levi-Civita connection
and volume density of $g(\tau)$ will be denoted by $\nabla =\nabla(\tau)$ and $d\mu = d\mu_{g(\tau)}$, respectively.
Let $\Xc$ and $\Yc$ be smooth vector bundles over $M$ equipped with their own smooth families of bundle metrics.
To avoid having to introduce separate notation for these metrics, we will regard
$\Xc$ and $\Yc$ as orthogonal subbundles of the bundle $\Wc = \Xc\oplus \Yc$ relative to the 
single family of metrics $\gamma = \gamma(\tau)$ on $\Wc$.
We will write $\pdtau g \dfn b$, $\pdtau \gamma \dfn \beta$, and $B \dfn \operatorname{tr}_g(b) = g^{ij}b_{ij}$. 

We will also assume that $\Xc$ is equipped with a smooth family of connections $\nabh = \nabh(\tau)$ compatible
with $\gamma(\tau)$, and use $\nabh$ to denote also the  family of connections
induced by $\nabla$ and $\nabh$ on tensor products of the bundles $TM$ and $T^*M$ with $\Xc$. With this convention,
we introduce the elliptic operator $\Box \dfn \Lambda^{ij}\nabh_i\nabh_j$ for
some smooth, symmetric positive definite family of sections $\Lambda = \Lambda(\tau)$ of $T^2_0(M)$.
Here and elsewhere we use the convention
that repeated Latin indices denote a sum from $1$ to $n$. Finally, we define the backward and forward parabolic operators
\[
  \Lc_B \dfn \pdtau + \Box + \nabla_{i}\Lambda^{ij}\nabla_j, \quad \Lc_F \dfn 
\pdtau - \Box - \nabla_{i}\Lambda^{ij}\nabla_j
\]
acting on smooth families of sections of $\Xc$.

\subsection{Integral identities}

The main ingredients of the proof of Theorem \ref{thm:bu2ndorder} below
are the integral identities and inequalities in the next two lemmas. Except
for some ultimately insignificant error terms arising from the time-dependency of the metrics
and connections involved, these identities are essentially specific instantiations of those in  Agmon-Nirenberg \cite{AgmonNirenberg1, AgmonNirenberg2} 
(see also \cite{Kukavica1}, \cite{Poon}). 

For $\tau \in [0, \Omega]$, we will let $(\cdot, \cdot )$ denote the $L^2(d\mu_{g(\tau)})$-inner product
induced by $\gamma(\tau)$ on $\Wc$,
that is,
\[
    (U, V) \dfn \int_{M}\langle U, V\rangle_{\gamma(\tau)}\,d\mu_{g(\tau)},
\]
and use $\|\cdot\|$ to denote the associated norm. To reduce the clutter in our expressions, we will use an unadorned norm $|\cdot|$
to denote all fiberwise norms induced by $g$ and $\gamma$ on the various bundles we will encounter below;
the context will determine the norm unambiguously.
For smooth compactly supported families of sections $X = X(\tau)$ and $Y= Y(\tau)$ of $\Xc$ and $\Yc$, respectively, we define 
\[
  \Ec \dfn \Ec(X, Y) \dfn \|X\|^2 + \|Y\|^2, \quad \Fc \dfn \Fc(X),
  \dfn \int_M |\nabh X|_{\Lambda}^2\,d\mu
\]
where $|\nabh X|_{\Lambda}^2 \dfn \Lambda^{ij}\langle\nabh_i X, \nabh_j X\rangle$.

\begin{lemma}\label{lem:intident}
 Suppose $X$ and $Y$ are smooth families of sections of $\mathcal{X}$ and $\mathcal{Y}$ on $[0, \Omega]$
such that, for all $\tau$, $\operatorname{supp}{X(\cdot, \tau)}$ and $\operatorname{supp}{Y(\cdot, \tau)}$
are contained in some fixed compact subset of $M$.
Then, 
\begin{align} 
\begin{split}\label{eq:l2ev} 
  \frac{d}{d\tau} \Ec(X, Y)
  &= 2\Fc(X) +  2\left(\Lc_B X, X\right) + 2\left(\partial_{\tau} Y, Y\right) + I_1(X, Y)\\
 &= \left(\Lc_B X, X\right) + \left(\Lc_F X, X\right) + 2\left(\partial_{\tau} Y, Y\right) +  I_1(X, Y),	 
\end{split}\\  
\begin{split}\label{eq:h1arr1}
  \mathcal{F}(X) &= \frac{1}{2}\left(\left( \Lc_F X, X) - (\Lc_B X, X\right)\right),
 \end{split}\\ 
\begin{split}\label{eq:h1ev}
 \frac{d}{d\tau}\Fc(X) &=
    \frac{1}{2}\left(\|\Lc_F X\|^2 - \|\Lc_B X\|^2\right)+ I_2(X, \nabh X),
\end{split}
\end{align}
for all $\tau \in [0, \Omega]$, where
\begin{align*}
    I_1(X, Y) &\dfn\int_M\left(\beta(X\oplus Y, X\oplus Y) + \frac{B}{2}(|X|^2+ |Y|^2)\right)\,d\mu,\\
\begin{split} 
 I_2(X, \nabh X) &\dfn \int_M\bigg(\pdtau\Lambda^{ij}\langle\nabh_i X, \nabh_j X\rangle
+ 2\Lambda^{ij}\left\langle\left[\pdtau, \nabh_i\right]X, \nabh_j X\right\rangle\\
&\phantom{=\int_M\bigg(}+ \Lambda^{ij}\beta(\nabh_i X, \nabh_j X) + \frac{B}{2}|\nabh X|_{\Lambda}^2 \bigg)\,d\mu.
\end{split}
\end{align*}
\end{lemma}
\begin{proof}
The identities in \eqref{eq:l2ev} follow from differentiating under the integral sign and writing the
term involving $\pdtau X$ in two ways. Writing
\begin{align*}
    \int_{M}2\left\langle\pdtau X, X\right\rangle\,d\mu &=
      \int_{M}2\left\langle\Lc_B X -\Box X -\nabla_i\Lambda^{ij}\nabh_jX , X\right\rangle\,d\mu
\end{align*}
and integrating by parts yields the first, while writing
\begin{align*}
 \int_{M}2\left\langle\pdtau X, X\right\rangle\,d\mu =\int_{M}\left\langle\Lc_B X + \Lc_FX, X\right\rangle\,d\mu
\end{align*}
yields the second. 

 For equation \eqref{eq:h1arr1}, we simply integrate by parts to obtain
\[
     \Fc(X) = -\int_{M}\left\langle \Box X + \nabla_i\Lambda^{ij}\nabh_j X, X \right\rangle\,d\mu
  = \frac{1}{2}\int_{M}\!\!\!\left(\langle \Lc_F X, X\rangle - \langle \Lc_B X, X\rangle\right)\,d\mu.
\]

 For equation \eqref{eq:h1ev}, we differentiate under the integral sign to obtain
\begin{align*}
\begin{split}
 \frac{d}{d\tau}\Fc(X)
 &= 2\int_M\Lambda^{ij}\left(\left\langle \nabh_i\pdtau X, \nabh_j X \right\rangle +
\Lambda^{ij}\left\langle\left[\pdtau, \nabh_i\right]X, \nabh_j X\right\rangle  \right)\,d\mu\\
  &\phantom{=} +\int_{M}\left(\pdtau\Lambda^{ij}\langle\nabh_i X, \nabh_j X\rangle+ \Lambda^{ij}\beta(\nabh_i X, \nabh_j X) 
+ \frac{B}{2}|X|^2 \right)\,d\mu.
\end{split}
\end{align*}
Rewriting the first term as
\begin{align*}
&2\int_M\Lambda^{ij}\left\langle \nabh_i\pdtau X , \nabh_j X \right\rangle\,d\mu\\
&\qquad = -2\left(\pdtau X, \Box X + \nabla_i\Lambda^{ij}\nabh_j X\right) 
= \frac{1}{2}\left(\|\Lc_F X\|^2 - \|\Lc_B X\|^2\right)
\end{align*}
then yields the desired identity.
\end{proof}

The key step in the proof of Theorem \ref{thm:bu2ndorder} backward uniqueness theorem 
is to bound the frequency/energy-quotient quantity
 $\Fc/\Ec$, which controls the rate of vanishing of $\log \Ec$.  
The following inequality (which follows \cite{AgmonNirenberg1}, \cite{AgmonNirenberg2}, \cite{Poon}) 
is the basis of this bound.
\begin{lemma}\label{lem:freqbound}
Suppose that $\Ec(\tau) > 0$ on $(a, b)$ and that
 $X(\cdot, \tau)$ and $Y(\cdot, \tau)$ have support in some fixed compact subset of $M$ for each $\tau$.
Then $\Nc = \Fc/\Ec$ satisfies
\begin{align}
\begin{split}\label{eq:nev}
   -\left(\frac{\|\Lc_B X\|^2 + \|\partial_{\tau}Y\|^2}{2\Ec}\right) + \frac{\Ic}{\Ec^2}\leq
\dot{\Nc} \leq \frac{\|\Lc_F X\|^2 + \|\partial_{\tau}Y\|^2}{2\Ec} + \frac{\Ic}{\Ec^2}\\
\end{split} 
\end{align}
where $\Ic \dfn I_2(X, \nabh X)\Ec - I_1(X, Y)\Fc$.
\end{lemma}
\begin{proof} Now, $\dot{\Nc} = (\dot{\Fc}\Ec - \dot{\Ec}\Fc)/\Ec^2$ and,
by \eqref{eq:h1ev}, we have
\begin{equation*}
 \dot{\Fc}\Ec = \frac{1}{2}\bigg\{\left(\|\mathcal{L}_FX\|^2 -\|\mathcal{L}_{B}X\|^2\right)\|X\|^2
				  +\left(\|\mathcal{L}_FX\|^2 -\|\mathcal{L}_{B}X\|^2\right)\|Y\|^2\bigg\}  + I_2\Ec,
\end{equation*}
while, by \eqref{eq:h1arr1} and \eqref{eq:h1ev}, we have
\begin{align*}
\begin{split}
\dot{\Ec}\Fc &= \frac{1}{2}\bigg\{(\Lc_F X, X) + (\Lc_B X, X) 
  + 2(\partial_{\tau}Y, Y)\bigg\}\cdot\bigg\{(\Lc_F X, X) - (\Lc_B X, X)\bigg\}
   + I_1\Fc \\
&= \frac{1}{2}\bigg\{(\Lc_F X, X)^2 - (\Lc_B X, X)^2\bigg\}
   + (\partial_{\tau}Y, Y)\bigg\{(\Lc_F X, X) - (\Lc_B X, X)\bigg\}+ I_1 \Fc.
\end{split} 
\end{align*}
Thus, taken together, we have
\begin{align}
\begin{split}\label{eq:nder}
\dot{\Nc} &= \frac{\|\Lc_F X\|^2\|X\|^2 - (\Lc_F X, X)^2}{2\Ec^2}
     + \frac{(\Lc_B X, X)^2 - \|\Lc_B X\|^2\|X\|^2}{2\Ec^2} \\
  &\phantom{=} + \left\{\frac{\|\Lc_F X\|^2 - \|\Lc_B X\|^2}{2\Ec^2}\right\}\|Y\|^2
	       + \left\{\frac{(\Lc_B X, X) - (\Lc_F X, X)}{\Ec^2}\right\}(\partial_{\tau}Y, Y)\\
  &\phantom{=}+ \frac{I_2\Ec - I_1\Fc}{\Ec^2}.
\end{split}
\end{align}
Now,
\begin{align}\label{eq:ytermsl1}
  \frac{1}{2}\|\Lc_F X\|^2\|Y\|^2 - (\Lc_F X, X) (\partial_{\tau} Y, Y) 
  & \geq -\frac{1}{2}\|\partial_{\tau} Y\|^2\|X\|^2,
\end{align}
and
\begin{align}\label{eq:ytermsl2}
    \frac{1}{2}(\Lc_{B}X, X)^2 + (\mathcal{L}_B X, X) (\partial_{\tau}Y, Y) \geq - \frac{1}{2}\|\partial_{\tau} Y\|^2\|Y\|^2,
\end{align}
thus, since the first term in \eqref{eq:nder} is nonnegative, we have
\begin{align*}
 \dot{\Nc} \geq -\frac{(\|\Lc_B X\|^2 + \|\partial_{\tau} Y\|^2)(\|X\|^2+\|Y\|^2)}{2\Ec^2}  + \frac{\Ic}{\Ec^2},
\end{align*}
which is the left-hand inequality in \eqref{eq:nev}.

The right-hand inequality in \eqref{eq:nev} can be obtained in the same way, using the non-positivity of the second term
in \eqref{eq:nder} and substituting the inequalities 
\[
  -\frac{1}{2}\|\Lc_B X\|^2\|Y\|^2 + (\Lc_B X, X) (\partial_{\tau} Y, Y) 
  \leq \frac{1}{2}\|\partial_{\tau} Y\|^2\|X\|^2,
\]
and
\[
 -\frac{1}{2}(\Lc_{F}X, X)^2 - (\mathcal{L}_F X, X) (\partial_{\tau}Y, Y) \leq \frac{1}{2}\|\partial_{\tau} Y\|^2\|Y\|^2,
\]
in place of  
\eqref{eq:ytermsl1} and \eqref{eq:ytermsl2}, respectively.
\end{proof}

\subsection{A general backward uniqueness theorem}
\label{ssec:2ndorderbu}
The following theorem is a slight generalization of one proven in \cite{KotschwarRFBU} using Carleman-type estimates. We now
give a simpler proof.  The notation
$\left[\pdtau, \nabh\right]$ denotes the family of sections $\pdtau \nabh - \nabh\pdtau  \in C^{\infty}(T^*M\otimes \operatorname{End}(\Xc))$.

\begin{theorem}[cf. \cite{KotschwarRFBU}, Theorem 3.1]\label{thm:bu2ndorder}
Assume that $(M, g(\tau))$ is complete for all $\tau\in [0, \Omega]$, and let $r_0(x) = d_{g(0)}(x, x_0)$
for some fixed $x_0\in M$. Suppose that $\Lambda(\tau) \in C^{\infty}(T^2_0(M))$ satisfies the lower bound $\lambda g^{ij}\leq \Lambda^{ij}$ for some $\lambda > 0$ and, together with 
$g$ and $\nabh$, the uniform bounds
\begin{equation}\label{eq:ubounds}
\sup_{M\times[0, \Omega]}\left\{|b| + |\nabla b| + |\Lambda|  + |\nabla \Lambda| + \left|\partial_{\tau}\Lambda\right|
  + \left|\left[\pdtau, \nabh\right]\right| + |\Rm| \right\} \leq L_0,
\end{equation}
for some $\lambda$, $L_0 > 0$. Then, if the smooth families of sections
$X(\tau) \in C^{\infty}(\Xc)$ and $Y(\tau) \in C^{\infty}(\Yc)$ 
vanish identically for $\tau = 0$, satisfy the growth condition
\begin{equation}\label{eq:xygrowth}
  \sup_{M\times[0, \Omega]} e^{-L_1(r_0(x) + 1)}\left\{|X| + |\nabla X| + |Y| \right\} <\infty,
\end{equation} 
and the system of inequalities
\begin{align}\label{eq:pdeode2nd}
\begin{split}
\left|\partial_{\tau}X + \Box X\right| &\leq C_0\left(|X| + |\nabla X| + |Y|\right),\quad
\left|\partial_{\tau}Y\right| \leq C_0\left(|X| + |\nabla X| + |Y|\right). 
\end{split}
\end{align}
on $M\times [0, \Omega]$, they must vanish identically.
\end{theorem}
We first give the proof in the technically simpler case that $M$ is compact.
\begin{proof}{\em (Compact Case.)}
The quantities $\Ec = \Ec(X, Y)$ and $\Fc = \Fc(X)$
are differentiable on $[0, \Omega]$ and $\Ec(0) = 0$. Consequently, the set
$A = \{\,\tau\in (0, \Omega]\,|\, \Ec(\tau) > 0\,\}$ is open; if  $A =\emptyset$, there is nothing to prove. 
Otherwise, writing $\alpha = \inf A$, there is $\omega > \alpha$ such that $\Ec(\tau) > 0$ on $(\alpha, \omega]$. 
Note that, since $\Ec(0) = 0$, we must have $\Ec(\alpha) = 0$.
On $(\alpha, \omega]$, we define $\Nc \dfn \Fc/\Ec$ as in the preceding section.

By Lemma \ref{lem:freqbound}, the derivative of $\Nc$ satisfies
\begin{equation}\label{eq:nevlower}
\dot{\Nc} \geq   - \left(\frac{\|\Lc_B X\|^2 + \|\partial_{\tau}Y\|^2}{2\Ec}\right) + \frac{\Ic}{\Ec^2},
\end{equation}
where $\Ic = I_2\Ec -I_1\Fc$ and $I_1$, $I_2$ are the error terms appearing in Lemma \ref{lem:intident}.  Since $M$ is compact, and $\|\nabh X\|^2 \leq \lambda^{-1}\Fc$,
there is a constant $C = C(\lambda, L_0)$ such that $I_2 \geq  -C(\Ec +\Fc)$ and $I_1 \leq C\Ec$ on $[0, \Omega]$.   
Thus, using \eqref{eq:pdeode2nd}, for a further enlarged $C > 0$
depending on $C_0$, we have
\[
 \dot{\Nc} \geq   -C(\Nc + 1).
\]
It follows that $\Nc(\tau) \leq e^{C\omega}(\Nc(\omega) + 1)\dfn N_0$ for all $\tau \in (\alpha, \omega]$.

On the other hand, using \eqref{eq:pdeode2nd} with the first identity in \eqref{eq:l2ev} and the compactness of $M$, we have
\[
    \dot{\Ec} \leq C\Ec + 2\Fc + \|\Lc_B X\|^2 + \|\partial_{\tau} Y\|^2 
  \leq C(\Ec + \Fc)
\]
on $(\alpha, \omega]$
for some $C= C(\lambda, C_0, L_0)$. It follows that $\dot{\widehat{\log \Ec}} \leq C(N_0+1)$, whence
$\Ec(\omega) \leq \Ec(\tau)e^{C(N_0+1)(\omega-\tau)}$ for any $\alpha < \tau \leq \omega$.  Sending $\tau \searrow \alpha$, 
we conclude that $\Ec(\alpha) > 0$, a contradiction. Therefore
$A =\emptyset$ and $\Ec(\tau) \equiv 0$ on $[0, \Omega]$.
\end{proof}

\subsection{Noncompact case}
The crux of the proof in this case is again Lemma \ref{lem:freqbound}; we will just need to cut-off the sections $X$ and $Y$
in space and condition them with a suitably rapidly decaying function to ensure the finiteness
of the various integral quantities and the validity of the manipulations we need to perform.

First, we observe that, under the assumptions of Theorem \ref{thm:bu2ndorder},
 $\Rc(g(0))$ is bounded below and $|b(x, \tau)|\leq L_0$ on $M\times [0, \Omega]$. Therefore, the metrics $g(\tau)$
are uniformly bounded and,
by the Bishop-Gromov theorem, 
there is a constant $V_0 = V_0(n, L_0, \Omega)$ such that 
\begin{equation}\label{eq:volgrowth}
    \operatorname{vol}_{g(\tau)}(B_{g_0}(x_0, r)) \leq V_0e^{V_0r}
\end{equation}
for any $x_0 \in M$, $r> 0$, and $\tau\in [0, \Omega]$. 

Next, we record a slight extension to a lemma from \cite{ChowRFV2P3} (cf. \cite{Tam}) concerning 
the existence of smooth distance-like functions with controlled derivatives 
up to a certain prescribed order.
\begin{lemma}\label{lem:distapprox} Let $m\geq 0$ and $(M, g(\tau))$ be a smooth family of complete Riemannian manifolds with 
$\pdtau g = b$. Assume that there is a constant $L$ such that 
\[
\sup_{M\times[0, \Omega]}\sum_{k=0}^{m-1}|\nabla^{(k)}b|\leq L, \quad
\sup_{M\times[0, \Omega]}\sum_{k=0}^{m-2}|\nabla^{(k)}\Rm|\leq L. 
\]
Then, there exists a smooth function $\rho: M\to (0, \infty)$
and positive constants $C_1 = C_1(n, L, \Omega)$, $C_2=C_2(n, m, L, \Omega)$ such that
\begin{align}
\label{eq:rhocomp}
  &C_1^{-1}(1 + r_0(x)) \leq \rho(x) \leq C_1(1 + r_0(x)), \quad \sup_{M\times[0, \Omega]} \sum_{k=0}^{m}|\nabla^{(k)} \rho| \leq C_2.
\end{align}
Here, as before, $r_0(x) = d_{g(0)}(x, x_0)$ for some fixed point $x_0$ of $M$, and $|\cdot| = |\cdot|_{g(\tau)}$,
$\nabla = \nabla_{g(\tau)}$, and $\Rm = \Rm(g(\tau))$.
\end{lemma}
\begin{proof}
  In Section 26.4 of \cite{ChowRFV2P3} it is shown that, under the assumptions on the curvature tensor alone,
there is a smooth function $\rho$ satisfying the first condition in \eqref{eq:rhocomp} and the second condition at $\tau = 0$. 
The uniform bound on $|b|$ then
implies the uniform equivalence of the metrics $g(\tau)$ and the bounds on $|\nabla\rho|_{g(\tau)}$ for all $\tau$. 
We can then control
$|\nabla^{(k)} \rho|_{g(\tau)}$ for $1 < k \leq m$ inductively with bounds
on $\nabla^{(l)}\rho$
and $\nabla^{(l)}b$ for $l\leq k-1$, since the latter quantities control the evolution of the connection.
\end{proof}

\subsubsection{Proof of Theorem \ref{thm:bu2ndorder} in the noncompact case} We will use a cut-off argument 
to bound a localized version of $\Nc = \Fc/\Ec$
with some additional error. Our 
assumptions do not, a priori, preclude the possibility that $X(\tau_0)$ and $Y(\tau_0)$ may vanish identically on some
open set at some $\tau_0$, so we must choose the size of the support of our cut-off function with a little care.

\begin{proof} Define $B_R \dfn B_{g(0)}(x_0, R)$ and let $\rho$ be the function guaranteed by
 Lemma \ref{lem:distapprox} with $m=2$. Using $\rho$ we can construct 
a cut-off function $\phi_R\in C^{\infty}_c(M, [0, 1])$ satisfying
\[
  \phi_R \equiv 1\ \mbox{on}\ B_{R}, \quad \supp \phi_R \subset B_{2R}, \quad 
\sup_{M\times[0, \Omega]}\left\{ |\nabla \phi_R| + |\nabla\nabla \phi_R| \right\}\leq C_3.
\]

As we have observed in \eqref{eq:volgrowth}, there exists a constant $V_0$ such that 
$\operatorname{vol}_{g(\tau)}(B_R)\leq V_0 e^{V_0R}$ for all $\tau$.  Define $B = \max\{L_1, V_0\}$,
and put
\[
\Xt \dfn e^{-3B\rho}X, \quad \Yt \dfn e^{-3B\rho}Y,\quad \Xt_{R} \dfn \phi_R\Xt,\quad  \Yt_{R} \dfn \phi_R\Yt.  
\]
Then, $\Xt_R(\cdot, 0) = 0$ and $\Yt_R(\cdot, 0) = 0$ for any $R > 0$,
\[
  \sup_{R > 0}\sup_{M\times[0, \Omega]} e^{2B\rho}\left\{|\Xt_R| + |\nabh\Xt_R| + |\Yt_R|\right\} < C_4,
\]
and
\begin{align}
\label{eq:xytbd}
\begin{split}
 \left|\partial_{\tau}\Xt + \Box \Xt_R\right| + \left|\partial_{\tau}\Yt_R\right| &\leq 
C(|\Xt_R| + |\nabh \Xt_R| + |\Yt_R|) + Ce^{-2BR}\chi_{\supp\nabla\phi_{R}}
\end{split}
\end{align}
for some $C$ which depends on $n$, $B$, $C_0$, $C_3$ and $C_4$, but is independent of $R$.

The energy quantities $\Ec \dfn \Ec(\Xt, \Yt)$, $\Fc \dfn \Fc(\Xt)$, $\Ec_{R} \dfn \Ec(\Xt_R, \Yt_R)$, 
and $\Fc_R \dfn \Fc(\Xt_R)$
are then well-defined and differentiable on $[0, \Omega]$. As before, the set $A = \{\tau\in (0, \Omega]\,|\,\Ec(\tau) > 0\}$
is open, and, if it is nonempty, we can find $0\leq \alpha < \omega \leq \Omega$ such that
$\Ec(\tau) > 0$ on $(\alpha, \omega]$ and $\Ec(\alpha) = 0$.

Now fix $a\in (\alpha, \omega)$.  We will continue to use $C$ to denote a sequence of large constants that may 
depend on the parameters in the assumptions of Theorem \ref{thm:bu2ndorder} but are independent
of $a$ and $R$. Note that there is $R_a\geq 1$ such that $\Ec_{R}(\tau) > 0$ for all $\tau\in [a, \omega]$
provided $R \geq R_a$.  Using \eqref{eq:xytbd} and that $\operatorname{vol}_{g(\tau)}(\supp \nabla\phi_{R}) \leq Ce^{2BR}$, 
we have
\begin{align*}
 \|\Lc_B \Xt_{R}\|^2 + \|\partial_{\tau}\Yt_{R}\|^2 &\leq C(\Ec_{R} + \Fc_{R}) + Ce^{-2BR}.
\end{align*}
Using this inequality and controlling $\Ic$ via \eqref{eq:ubounds} as in the compact case, Lemma \ref{lem:freqbound} implies that
$\Nc_{R} \dfn \Fc_{R}/ \Ec_{R}$ satisfies
\[
 \dot{\Nc}_{R} \geq -C(\Nc_{R} + 1) + Ce^{-2BR}\Ec_{R}^{-1}.
\]
and hence that
\begin{align*}
  \frac{d}{d\tau}\left(e^{C\tau}(\Nc_{R}(\tau)+ 1)\right) \geq - Ce^{C\tau-2BR}\Ec_{R}^{-1}(\tau),
\end{align*}
on $[a, \omega]$ for all $R \geq R_a$.  Since $\Nc_{R}(\omega)\to \Nc(\omega)$ as $R\to\infty$, 
we have $\Nc_{R}(\omega) \leq N_0$ for some $N_0$ and this implies the estimate
\begin{align*}
  \Nc_{R}(\tau) + 1 &\leq e^{C(\omega - \tau)}\left(\Nc_{R}(\omega) + 1\right)
	+ Ce^{-2BR}\int_{\tau}^{\omega}e^{C(s-\tau)}\Ec_{R}^{-1}(s)\,ds\\
    &\leq C(N_0 +1) + Ce^{-2BR}\int_{\tau}^{\omega}\Ec_{R}^{-1}(s)\,ds,
\end{align*}
for $\tau$ and $R$ in the same range.  We must control the new error term.  Define 
$Q_R(\tau) \dfn \int_{\tau}^{\omega}\Ec_{R}^{-1}(s)\,ds$, and observe that $Q_R(\tau) \leq Q_R(a) < \infty$  and
\[
  \lim_{R\to\infty} Q_R(\tau) = \int_{\tau}^{\omega}\Ec^{-1}(s)\,ds < \infty
\]
on $[a, \omega]$ for all $R\geq R_a$.  Also, define $P_{R, a} \dfn N_0 + 1 +  e^{-2BR}Q_R(a)$
and note for later that $\lim_{R\to\infty} P_{R, a} = N_0 +1$, a finite nonzero limit independent of $a$.

If we apply the first identity in \eqref{eq:l2ev} to $\dot{\Ec}_{R}$
and use \eqref{eq:xytbd} again with our above inequality for $\dot{\Nc}_{R}$,
we then have
\begin{align*}
\frac{d}{d\tau}\bigg\{e^{-CP_{R, a}\tau}\Ec_{R}(\tau)\bigg\} 
\leq Ce^{-CP_{R, a}\tau - 2BR}
\end{align*}
on $[a, \omega]$ and hence 
\[
  e^{-CP_{R, a}(\omega -a)} \Ec_{R}(\omega) \leq \Ec_{R}(a) + \frac{e^{-2BR}}{P_{R, a}}\left(1 - e^{-CP_{R, a}(\omega -a)}\right).
\]
for all $R \geq R_a$.
Sending $R\to\infty$, we obtain
\[
 e^{-C(N_0 + 1)(\omega-a)}\Ec(\omega) \leq \Ec(a)\,
\]
which implies  $\Ec(\omega) = 0$ upon sending $a\searrow \alpha$. Therefore we must have $A = \emptyset$.
\end{proof}

\section{Applications to geometric flows: the prolongation procedure}
\label{sec:2ndorderapps}

As we have already mentioned, the lack of short-time existence for backward parabolic problems
means that the problem of backward uniqueness for geometric flows such as the Ricci flow cannot be simply reduced
to one for an associated strictly parabolic system by a standard implementation of the DeTurck trick.
In our prior work on the Ricci flow \cite{KotschwarRFBU}, \cite{KotschwarRFHolonomy}, and building on an idea of 
Alexakis \cite{Alexakis} (see also \cite{WongYu}), we demonstrated that, appropriately reformulated, the problem can nevertheless be attacked
by essentially the same methods for strictly parabolic equations.

The philosophy behind this reformulation is that the curvature tensors (and their covariant derivatives) associated
to the solutions of geometric flows often themselves satisfy strictly parabolic equations.  
Since the elliptic operators in these curvature equations typically depend in a nonlinear fashion on the solution metrics,
the difference of the curvature tensors associated to two different solutions
will not, in general, satisfy a strictly parabolic equation itself. Nevertheless, 
it is possible in many cases to add logically redundant
lower-order differences of the metric and higher-order differences of the covariant derivatives of the curvature tensors to obtain
a prolonged system which satisfies a closed and effectively parabolic system of inequalities.
We illustrate this procedure now on three examples: on the Ricci and cross-curvature flows in this section,
and on the $L^2$-curvature flow in the next.
\subsection{The Ricci flow} 
Assume that $g$ and $\gt$ are complete solutions of uniformly bounded curvature to the Ricci flow
\begin{equation}\label{eq:rf}
\pdt g = -2\Rc(g) 
\end{equation}
on $M\times [0, \Omega]$ which agree at $t=0$. For our first example, we will review the construction in Section 2 of \cite{KotschwarRFBU}. 
It will be convenient to define $\tau = \Omega - t$ and consider $g= g(\tau)$ and $\gt = \gt(\tau)$ as solutions
to the backward Ricci flow with $g(0) = \gt(0)$.

The difference $X^{(k)} = \nabla^{(k)}\Rm - \nabt^{(k)}\Rmt$ of the curvature tensors will satisfy an equation
of the form
\[
    \left(\pdtau + \Delta\right)X^{(k)}  = (\Deltat - \Delta)\nabt^{(k)}\Rmt + P^k_2(\Rm) - \tilde{P}^k_2(\Rmt)
\]
where $P^k_l(U)$ (respectively, $\tilde{P}^k_l(U)$) represents a linear combination of terms
formed by contraction by $g$ (respectively, $\gt$) from tensor products of the form
 $\nabla^{(i_1)}U\otimes\nabla^{(i_2)}U\otimes \cdots \otimes \nabla^{(i_l)}U$ (respectively, 
$\nabt^{(i_1)}U\otimes\nabt^{(i_2)}U\otimes \cdots \otimes \nabt^{(i_l)}U$) where
$i_1 + i_2 + \cdots + i_l = k$.  

The key observation is that, \emph{for all $k$}, the first term on the right-hand side
can be expressed as a polynomial of the form
\[
  (\Delta - \Deltat)\nabla^{(k)}\Rmt = Y^{(0)} \ast_g T_0 + Y^{(1)} \ast_g T_0 + Y^{(2)} \ast_g T_3
\]
where $Y^{(0)} = g - \gt$, $Y^{(1)} = \nabla - \nabt$, $Y^{(2)} = \nabla Y^{(2)}$ and the $T_i$ are (bounded) catch-all factors
involving $Y^{(1)}$ and $\nabt^{(l)}\Rmt$ for $l=0, 1, 2, \ldots, k+2$.    The quantities
$Y^{(i)}$, in turn, for $i =0, 1, 2$, can be controlled in an ordinary differential sense by $X^{(0)}$, $X^{(1)}$
and $\nabla X^{(1)}$. Thus we add $X^{(1)}$, $Y^{(0)}$, $Y^{(1)}$, and $Y^{(2)}$ to our original system consisting of
$X^{(0)}$ to obtain a closed system of inequalities.

This is made precise in \cite{KotschwarRFBU}, where it is shown that, if two solutions $g$ and $\gt$ agree at $\tau= 0$
then, for each $\delta > 0$, the family of sections $X = (X^{(0)}, X^{(1)})$ and $Y = (Y^{(0)}, Y^{(1)}, Y^{(2)})$ 
of $\Xc\dfn T_3^1(M)\oplus T_4^1(M)$ and $\Yc \dfn T_{2}^0(M)\oplus T_{2}^1(M)\oplus T_{3}^1(M)$,
together satisfy the system
\begin{align*}
  \left|\partial_{\tau}X + \Delta X\right| \leq C_{\delta}(|X| + |Y|), \quad
  \left|\partial_{\tau}Y\right| \leq C_{\delta}(|X| +|\nabla X| + |Y|)
\end{align*}
on any $[0, \Omega-\delta]$.  Here the dependence of the constants $C_{\delta}$ on $\delta$
is an artifact of the use of Shi's estimates \cite{Shi} to bound the coefficients (and can be eliminated if $M$ is compact).  
These coefficients depend on the derivatives of the curvature tensors of the solutions of up to third-order,
and under the assumption of bounded curvature, Shi's estimates ensure that these derivatives are bounded
uniformly for $t\in (\delta, \Omega]$, that is, for $\tau \in [0, \Omega -\delta)$.  To prove that $g \equiv \gt$, 
we apply Theorem \ref{thm:bu2ndorder} to the above system
on $M\times [0, \Omega -\delta]$ for any fixed $\delta$, and then send $\delta\searrow 0$.  We obtain the following
result.
\begin{theorem}[Theorem 1.1, \cite{KotschwarRFBU}]
 If $g$ and $\gt$ are complete solutions to \eqref{eq:rf} with uniformly bounded curvature,
then $g(\Omega) = \gt(\Omega)$ implies that $g \equiv \gt$ on $[0, \Omega]$.
\end{theorem}

In \cite{KotschwarRFTA}, we prove that a complete solution to the Ricci flow
of bounded curvature is in fact \emph{analytic} in time for $t > 0$ (see also \cite{Shao1} for a related result). 

\subsection{Cross-curvature flow}
On a closed Riemannian three-manifold $(M^3, g)$ of sectional curvature of negative (respectively positive) sign, the Einstein
tensor $E_{ij} = R_{ij} -(R/2)g_{ij}$  is negative (respectively positive) definite. 
Write $E^{ij} = g^{ik}g^{jl}E_{kl}$.  If $V= V_{ij}$ is the two-tensor satisfying $V_{ik}E^{kj} = \delta^{j}_i$ and $P =\det(E^i_j)$, the \emph{cross-curvature} tensor
of $g$ is defined to be the tensor
\[
 \varXi_{ij} \dfn PV_{ij} =  -(1/2)E^{pq}R_{pijq}.
\]
For such manifolds, Hamilton and Chow \cite{ChowHamiltonXCF} define the \emph{cross-curvature flow}
as the evolution equation
\begin{align}
  \label{eq:xcf}
      \partial_t g &= -2\sigma \varXi(g), 
\end{align}
for a family of metrics $g(t)$ where $\sigma = \pm 1$ depending as the sectional curvature of $g(0)$ is positive
or negative.  Short-time existence for the flow is proven
in \cite{Buckland}.

As an equation for the metric, \eqref{eq:xcf} is fully nonlinear, however, the parabolic portion of the prolonged system we derive
will be effectively quasilinear. For our purposes, the key is that the curvature quantities
associated to the solutions satisfy strictly parabolic equations with respect to the family of strongly elliptic operators 
$\Box\dfn -\sigma E^{ij}\nabla_i\nabla_j$. (Since $\nabla_iE^{ij} = 0$, these operators are even self-adjoint.)
 In particular (see Lemma 17 of \cite{KotschwarHOCFU}), the tensor $V$ satisfies the equation
\[
 \partial_t V_{ij}  =  \Box V_{ij} +  (E^{al}E^{kb} - 2E^{ab}E^{kl})\nabla_kV_{ai}\nabla_l V_{bj} 
- Pg^{kl}(V_{ik}V_{jl} + V_{ij}V_{kl}).
\]
with analogous equations for its covariant derivatives $\nabla^{(k)}V$.

Neglecting the constant $\sigma$, 
the difference of the operators $\Box$ and $\widetilde{\Box}$ satisfy a formula of the form
\begin{align*}
 (\Box - \widetilde{\Box})W &= E^{ab}\nabla_a\left((\nabla_b -\nabt_b)W\right)
	  +E^{ab}(\nabla_a  -\nabt_a)\nabt_b W + (E^{ab} - \Et^{ab})\nabt_a\nabt_bW\\
	&= Y^{(1)} \ast T_1 + Y^{(2)}\ast T_2 + X^{(0)} \ast T_3
\end{align*}
where $X^{(0)}\dfn V - \Vt$, $Y^{(1)} \dfn \nabla -\nabt$, $Y^{(2)} \dfn \nabla Y^{(1)}$ and the $T_i$
denote bounded tensors depending on $g$, $\gt$ and $W$.  (For the third term on the right, 
we have used the identity $E^{ab} - \Et^{ab}= -E^{ak}\Et^{bl}V_{kl}$.)

Defining $X^{(1)} \dfn \nabla V - \nabt\Vt$ and $Y^{(0)} \dfn g - \gt$, and carrying out computations
either contained in or similar to those in our paper \cite{KotschwarHOCFU}, one can verify that there is a constant $C$ such that 
$X = (X^{(0)}, X^{(1)})$ and $Y = (Y^{(0)}, Y^{(1)}, Y^{(2)})$ satisfy
\[
  \left|\partial_{\tau}X + \Box X\right| \leq C(|X| + |\nabla X| +  |Y|), \quad
  \left|\partial_{\tau} Y\right| \leq C(|X| +|\nabla X| + |Y|)
\]
on $M\times [0, \Omega]$. The following theorem is thus a consequence of Theorem \ref{thm:bu2ndorder} above,
taking $\Lambda^{ij} = -\sigma E^{ij}$.

\begin{theorem}
 If $g$, $\gt$ are solutions to \eqref{eq:xcf} on $M^3\times[0, \Omega]$ with either positive or negative
sectional curvature and $g(\Omega) = \gt(\Omega)$, then $g(t) = \gt(t)$ for all $t \in [0, \Omega]$.  
\end{theorem}

\section{The $L^2$-curvature flow and higher-order equations} \label{sec:l2}

The $L^2$-curvature flow, introduced by J.\ Streets in \cite{StreetsL2}, is the negative gradient flow of the 
functional $g\mapsto \int_M |\Rm(g)|^2\,d\mu_g$. The metric evolves according to the fourth-order equation
 \begin{equation}\label{eq:l2}
   \partial_t g_{ij} = 2\left(\Delta R_{ij} -\frac{1}{2}\nabla_i\nabla_jR\right) + 2R^{pq}R_{ipqj} -R^p_iR_{pj} + R_{i}^{pqr}R_{jpqr} 
-\frac{1}{4}|\Rm|^2g_{ij},
 \end{equation}
and the curvature tensor and its covariant derivatives satisfy a strictly parabolic equation of the form
\begin{equation}\label{eq:l2curvev}
\partial_t \nabla^{(k)}\Rm = -\Delta^{(2)} \nabla^{(k)} \Rm + P_2^{k+2}(\Rm) + P^{k}_3(\Rm). 
\end{equation}
See Proposition 4.3 in \cite{StreetsL2}. For our purposes, the key aspect of this equation
is that its principal part, the tensor bi-Laplacian, is a symmetric operator. Moreover, the following
result of Streets \cite{Streets4th} guarantees the instantaneous boundedness of all derivatives of the curvature under
the assumption of a uniform curvature bound.
\begin{theorem*}[\cite{Streets4th}, Theorem 1.3]
 If $(M^n, g(t))$ is a complete solution to \eqref{eq:l2}
 with $|\Rm|\leq K$ on $M\times [0, \Omega]$, then, for any $m\geq 1$,
\begin{equation}\label{eq:l2derbound}
  \sup_{M\times[0, \Omega]}t^{\frac{m+2}{4}}|\nabla^{(m)}\Rm|_{g(t)} \leq C
\end{equation}
for some $C = C(m, n, K)$.
\end{theorem*}

The main result of this section is the following theorem.
\begin{theorem}\label{thm:l2bu}
Suppose $g(t)$ and $\gt(t)$ are smooth complete solutions to \eqref{eq:l2}
on $M\times [0, \Omega]$ of uniformly bounded curvature.  If $g(\Omega) = \gt(\Omega)$,
 then $g\equiv \gt$ for all $t\in [0, \Omega]$.
\end{theorem}

\subsection{Prolongation}

Given two complete solutions $g$ and $\gt$ of bounded curvature on $M\times [0, \Omega]$ which agree at 
$t=\Omega$, the prolongation procedure can be carried out 
essentially as in the preceding sections. The difference of the bi-Laplacians $\Delta^{(2)}$ and $\Deltat^{(2)}$
is controlled by the difference $g-\gt$ of the metrics and its covariant derivatives 
of up to fourth order. The evolutions of these quantities, in turn, are controlled by the difference of the covariant derivatives
of $\Rm$ and $\Rmt$ of up to sixth order. 

Specifically, if we define $X^{(k)} \dfn \nabla^{(k)}\Rm -\nabt^{(k)}\Rmt$
and $Y^{(0)} \dfn g - \gt$, $Y^{(1)} \dfn \nabla - \nabt$, and $Y^{(k)} \dfn \nabla^{k-1}Y^{(1)}$ for $k=2, 3, \ldots$,
then the estimate \eqref{eq:l2derbound} and a computation (see, e.g., Lemma 7 of \cite{KotschwarHOCFU}) shows that, for all 
$0 < \delta < \Omega$, the sections 
$X \dfn (X^{(0)}, X^{(1)}, \ldots, X^{(4)})$ and $Y \dfn (Y^{(0)}, Y^{(1)}, \ldots, Y^{(4)})$
satisfy
\begin{equation}\label{eq:l2pdeode}
|\partial_{\tau}X - \Delta^{(2)} X| 
+ \left|\partial_{\tau} Y\right| \leq C\sum_{p=0}^2|\nabla^{(p)} X| + C|Y|,
\end{equation}
for some constant $C = C(\delta)$ on $(\delta, \Omega]$. 

\subsection{Proof of Theorem \ref{thm:l2bu}}
The argument is essentially the same as that for Theorem \ref{thm:bu2ndorder}, so we will focus here on the differences.
We first fix $\delta \in (0, \Omega)$, so that, in view of the bounds \eqref{eq:l2derbound},
the curvature tensors of $g$ and $\gt$ and all of their derivatives are uniformly bounded for $\tau \in [0, \Omega^{\prime}]$
where $\Omega^{\prime} =\Omega - \delta$. We will work on this interval, using 
$g(\tau)$ as a reference metric and defining $X$ and $Y$ as in the preceding section.

Under our assumptions, thanks to \eqref{eq:l2derbound}, the speed of the family of metrics $g(\tau)$ is uniformly bounded,
so the metrics $g(\tau)$ are uniformly equivalent on $[0, \Omega^{\prime}]$ and equation \eqref{eq:volgrowth} 
is therefore again satisfied for some $V_0 > 0$.  We can also again apply Lemma \ref{lem:distapprox} to obtain
 a smooth, proper exhaustion function $\rho$
  with uniform bounds on its $\nabla_{g(\tau)}$-covariant derivatives of order up to $4$.  Then, if $M$ is noncompact,
we can replace $X$ and $Y$ with $e^{-B\rho}X$ and $e^{-B\rho}Y$ where $B = B(n, V_0) > 0$ is chosen sufficiently large
to ensure that $X$ and $Y$ and all of their (bounded) derivatives are
$L^2(d\mu_{g(\tau)})$-integrable for each $\tau$.  This will ensure the validity of the integral operations
we will perform in what follows.  Moreover, equation \eqref{eq:l2pdeode} (with an enlarged constant $C$)
 will still hold for these modified $X$
and $Y$.

Next, we define the quantities
\[
 \Ec(X, Y) \dfn \|X\|^2 + \|Y\|^2,\quad \Fc \dfn \|\Delta X\|^2, \quad\Nc \dfn \Fc/\Ec
\]
(the latter being defined only when $\Ec > 0$) and the backward and forward parabolic operators
\[
  \Lc_B \dfn \pdtau -\Delta^{(2)}, \quad \Lc_F \dfn \pdtau + \Delta^{(2)}.
\]

A computation analogous to the first identity in \eqref{eq:l2ev}, using the boundedness of $\partial_{\tau} g$ together
with equation \eqref{eq:l2pdeode}, yields
\begin{align}\label{eq:l2eev1}
 \dot{\Ec} &\leq C\Ec + 2\Fc + C\sum_{p=0}^2\|\nabla^{(p)}X\|^2  
\end{align}
for some $C$. 

We will need two basic inequalities to proceed. First, with an induction argument using the identity 
$\|\nabla^{(l)} X\|^2 = -(\nabla^{(l+1)}X, \nabla^{l-1}X)$, we obtain 
the following simple interpolation statement: for all $\epsilon > 0$, 
and for all $0 \leq l < k$, there exists $C = C(\epsilon, k, l)$ such that
\begin{equation}\label{eq:gn}
     \|\nabla^{(l)} X\|^2 \leq C\|X\|^2 + \epsilon \|\nabla^{(k)}X\|^2
\end{equation}
for all $\tau\in [0, \Omega^{\prime}]$.
Second, the commutation formula $\left[\Delta, \nabla\right]X = \nabla X \ast \Rm + X \ast\nabla\Rm$
implies that
\[
  \|\Delta X\|^2 = -(\nabla_a X, \nabla_a \Delta X) = \|\nabla^{(2)}X\|^2 + (\nabla X \ast \Rm + X\ast \nabla \Rm, \nabla X),     
\]
and hence, with \eqref{eq:gn}, that, for any $\epsilon > 0$, there is a constant $C= C(\epsilon)$ such that
\begin{equation}\label{eq:gaarding}
  -C\|X\|^2 + (1-\epsilon)|\nabla^{(2)}X\|^2 \leq \|\Delta X\|^2 \leq (1+\epsilon)\|\nabla^{(2)} X\|^2 + C\|X\|^2
\end{equation}
for all $\tau\in [0, \Omega^{\prime}]$. Applying these inequalities to \eqref{eq:l2eev1}, we obtain that $\dot{\Ec} \leq C(\Ec + \Fc)$
on $[0, \Omega^{\prime}]$.  

Assume now that $A = \{\,\tau\in (0, \Omega^{\prime}]\,|\,\Ec(\tau) > 0\,\}$ is nonempty.
As before, there are $0 \leq \alpha < \omega \leq \Omega^{\prime}$ such that $\Ec(\alpha) = 0$ and $\Ec(\tau) > 0$
for all $\tau \in (\alpha, \omega]$. We have
$\dot{\widehat{\log \Ec}}\leq C(\Nc + 1)$ on $(\alpha, \omega]$, so the problem again is to bound $\Nc(\tau)$.

Since $\Delta^{(2)}$ is self-adjoint, the terms in the derivative $\Nc$ can be organized exactly
as in the proof of Theorem \ref{thm:bu2ndorder}. The only new terms are those
coming from $([\partial_{\tau}, \Delta]X, X)$ in the equation for $\dot{\Fc}$,
but these can be controlled using \eqref{eq:l2derbound} and the inequalities \eqref{eq:gn}, \eqref{eq:gaarding}:
\begin{align*}
  \left(\left[\pdtau, \Delta\right]X, \Delta X\right) &\geq -C\|\Delta X\|(\|X\| + \|\nabla X\| +\|\nabla^{(2)} X\|)\\
  &\geq -C(\|X\|^2 + \|\Delta X\|^2).
\end{align*}
Thus we obtain
\begin{align*}
 \dot{\Ec} & \leq C\Ec + (\Lc_F X, X) + (\Lc_B X, X) + 2(\partial_{\tau} Y, Y)\\
 \Fc &= \frac{1}{2}\left((\Lc_F X, X) - (\Lc_B X, X)\right)\\
 \dot{\Fc} & \geq - C(\Ec + \Fc) + \frac{1}{2}\left(\|\Lc_F X\|^2 - \|\Lc_B X\|^2\right)
\end{align*}
and, therefore, again a differential inequality of the form
\[
  \dot{\Nc} = \frac{\dot{\Fc}\Ec - \dot{\Ec}\Fc}{\Ec^2} \geq -C(\Nc+1) - \frac{1}{2\Ec}\left(\|\Lc_B X\|^2 + \|\partial_{\tau} Y\|^2\right).
\]
in which, using \eqref{eq:gn}, \eqref{eq:gaarding}, and the structural condition \eqref{eq:l2pdeode}, the last
term can be bounded below by $-C(\Nc + 1)$. The rest of the argument is identical: we obtain that $\Nc(\tau) \leq 
C(\Nc(\omega)+ 1) \dfn N_0$ and
conclude that $\Ec(\omega) \leq e^{C(N_0+1)(\omega-\tau)}\Ec(\tau)$ on $(\alpha, \omega]$. Sending $\tau\searrow \alpha$
implies that $\Ec(\omega) = 0$, from which we conclude that $A = \emptyset$, i.e., $\Ec\equiv 0$ on $[0, \Omega^{\prime}]$.
The theorem follows by sending $\delta \searrow 0$.

\subsection{Other higher-order equations}
The proof we have above extends easily, for example, to equations of the form
\begin{equation}\label{eq:kcf}
  \partial_t g = 2(-1)^{k+1}\Delta^{(k)}\Rc + P^{k-1}(\Rm)
\end{equation}
for $k\geq 1$ on compact manifolds (or, in the noncompact case, to complete solutions with appropriate bounds on the
derivatives of curvature). Here $P^r(\Rm)$ denotes a polynomial expression consisting of a finite number
of terms of the form $P_m^l(\Rm)$ with $l\leq r$ and $m\geq 0$.
The well-posedness of these (and a much broader class of equations) has been considered in 
\cite{BahuaudHelliwellSTE, BahuaudHelliwellU}. We sketch the argument below.

The curvature tensor of a solution $g$ to \eqref{eq:kcf} will evolve according to an equation of the form
\[
 \partial_t \Rm = (-1)^k\Delta^{(k+1)}\Rm + P^{k+1}(\Rm),
\]
where it is convenient, but not crucial, for us that the leading part of the equation is exactly a power of the tensor Laplacian.
The symmetry and strong ellipticity of the operator is all that we use. (These two conditions even
can be weakened somewhat further: see, e.g., \cite{AgmonNirenberg1}, \cite{AgmonNirenberg2}, \cite{Kukavica2}).
 
Given two solutions $g(\tau)$ and $\gt(\tau)$ to \eqref{eq:kcf} for $\tau\in [0, \Omega]$ on a closed manifold $M$ (or satisfying otherwise
suitable bounds on the derivatives of curvature), and defining $X^{(l)}$ and $Y^{(l)}$ as before, the sections
 $X\dfn (X^{(0)}, X^{(1)}, \ldots, X^{(3k+1)})$ and $Y \dfn (Y^{(0)}, Y^{(1)}, \ldots, Y^{2k+2})$ will satisfy
the following analog of \eqref{eq:l2pdeode},
\begin{equation}\label{eq:kcfpdeode}
 \left|\partial_{\tau}X + (-1)^k\Delta^{(k+1)}X\right| + \left|\partial_{\tau}Y\right|\leq C\left(\sum_{p=0}^{k+1}|\nabla^{(p)} X| + 
 |Y|\right).
\end{equation}
Defining
\[
\Ec(X, Y) \dfn \|X\|^2 + \|Y\|^2, \quad \Fc(X) \dfn \left\{\begin{array}{rl}
							  \|\Delta^{(m)} X\|^2 & \mbox{if}\ k = 2m-1\\
								\|\nabla\Delta^{(m)} X\|^2& \mbox{if}\ k = 2m
                                               \end{array} \right.,
\]
one can then argue as before to concude that, if $\Ec$ does not vanish on $(0, \Omega]$, there is a constant
$C$ such that $\Ec(\tau_2) \leq e^{C(\tau_2 - \tau_1)}\Ec(\tau_1)$ for all $0< \tau_1 \leq \tau_2 \leq \Omega$.  
In fact, using
the interpolation inequalities, one can obtain a similar inequality for the simpler quantity
$\tilde{\Ec} \dfn \|X^{(0)}\|^2 + \|X^{(3k+1)}\|^2 + \|Y^{(0)}\|^2 + \|Y^{(2k+2)}\|^2$.


\begin{thebibliography}{A}

 \bibitem[AN1]{AgmonNirenberg1} S.\ Agmon and L.\ Nirenberg,
  \textit{Properties of solutions of ordinary differential equaitons in Banach Space},
  Comm.\ Pure Appl.\ Math.\ {\bf 16} (1963), 121--239.

\bibitem[AN2]{AgmonNirenberg2} S.\ Agmon and L.\ Nirenberg,
 \textit{Lower bounds and uniqueness theorems for solutions of differential equations in a Hilbert space},  
  Comm.\ Pure Appl.\ Math.\ {\bf 20}  (1967), 207--229. 

 \bibitem[Al]{Alexakis} S. Alexakis,
   \textit{Unique continuation for the vacuum Einstein equations},
    Preprint (2009), {\tt arXiv:0902.1131}.
%
%

 \bibitem[BH1]{BahuaudHelliwellSTE} E.\ Bahuaud and D.\ Helliwell,
 \textit{Short-time existence for some higher-order geometric flows},
 Comm.\ P.\ D.\ E.\ \textbf{36} (2011), no. 12, 2189--2207. 
%
 \bibitem[BH2]{BahuaudHelliwellU} E.\ Bahuaud and D.\ Helliwell,
 \textit{Uniqueness for some higher-order geometric flows},
  Preprint (2014), {\tt arXiv:1407.4406 [math.DG]}.

%

\bibitem[BT]{BardosTartar} C.\ Bardos and L.\ Tartar,
  \textit{Sur l' unicit\'e r\'etrograde des \'equations paraboliques et quelques
questions voisines}, Arch.\ Rational Mech.\ Anal.\ {\bf 50} 1973, 10--25.


\bibitem[Bu]{Buckland} J.\ Buckland, 
  \textit{Short-time existence of solutions to the cross curvature flow on 3-manifolds}, 
Proc. Amer. Math. Soc. \textbf{134}, no. 6 (2006), 1803--1807.

%
%
 \bibitem[CZ]{ChenZhu}  B.\ -L.\ Chen and  X.\ -P.\ Zhu,
   \textit{Uniqueness of the Ricci flow on complete noncompact manifolds},
   J.\ Diff.\ Geom.\ {\bf 74}  (2006),  no. 1, 119--154.



 \bibitem[ChRF]{ChowRFV2P3}  B. Chow, S.\ -C.\ Chu, D.\ Glickenstein, C.\ Guenther, J.\ Isenberg,
 T.\ Ivey, D.\ Knopf, P.\ Lu, F.\ Luo, and L.\ Ni, 
  \textit{The Ricci flow: techniques and applications. Part III. Geometric-analytic aspects,}
  Mathematical Surveys and Monographs, {\bf 163}, 
 American Mathematical Society, Providence, RI, 2010.
 

\bibitem[CH]{ChowHamiltonXCF} B.\ Chow and R.\ S.\ Hamilton,
\textit{The cross curvature flow of 3-manifolds with negative sectional curvature},
Turkish J. Math. \textit{28} (2004), no. 1, 1--10.

%
%
%
%
%
%

\bibitem[E]{Escauriaza} L.\ Escauriaza,
  \textit{Carleman inequalities and the heat operator}, 
    Duke J.\ Math, {\bf 104} (2000), 113--127.

 \bibitem[EF]{EscauriazaFernandez} L.\ Escauriaza and F.\ -J.\ Fern\'andez,
 	\textit{Unique continuation for parabolic operators,}
 	Ark. Mat. {\bf 41} (2003), no. 1, 35--60. 
%
\bibitem[EKPV]{EKPV} L.\ Escauriaza, C.\ E.\ Kenig, G.\ Ponce and L.\ Vega, 
 	\textit{Decay at infinity of caloric functions within characteristic hyperplanes},
 	Math. Res. Lett. {\bf 13}  (2006),  no. 2-3, 441--453. 

 
 \bibitem[ESS]{ESS}  L.\ Escauriaza, G.\ Seregin, and V.\ Sver\'ak, 
 	\textit{Backward uniqueness for parabolic equations},
 	 Arch. Ration. Mech. Anal. {\bf 169}  (2003),  no. 2, 147--157. 

\bibitem[EV]{EV} L.\ Escauriaza and L.\ Vega, 
    \textit{Carleman inequalities and the heat operator. II}, 
  Indiana Univ.\ Math. J. {\bf 50} (2001), 1149--1169.


\bibitem[F]{Fernandez} F.\ J.\ Fernandez, 
\textit{Unique continuation for parabolic operators. II}, 
  Comm. P.\ D.\ E.\ {\bf 28} (2003), no. 9-10, 1597--1604.

 
\bibitem[G]{Ghidaglia} J.\ -M.\ Ghidaglia,
\textit{Some backward uniqueness results}\,
Nonlinear Anal. 10 (1986), no. 8, 777--790. 

%
%
%
%
\bibitem[KT]{KochTataru} H.\ Koch and D. Tataru,
  \textit{Carleman estimates and unique continuation for second order parabolic equations with nonsmooth coefficients},
    Comm. P.\ D.\ E.\ {\bf 34} (2009), no. 4-6, 305--366. 

\bibitem[H]{HamiltonSingularities} R.\ S.\ Hamilton,  
  ``The formation of singularities in the Ricci flow'', 
     \textit{Surveys in differential geometry, Vol. II}, 
       (Cambridge, MA, 1993), 7--136, Int. Press, Cambridge, MA, 1995.
%
%
\bibitem[Ko1]{KotschwarRFBU} B.\ Kotschwar,
   \textit{Backwards uniqueness for the Ricci flow},
   Int. Math. Res. Not. (2010), no. 21, 4064 -- 4097.


\bibitem[Ko2]{KotschwarRFHolonomy} B.\ Kotschwar,
\textit{Ricci flow and the holonomy group},
J. Reine Angew. Math. {\bf 690} (2014), 133--161. 

\bibitem[Ko3]{KotschwarRFTA} B.\ Kotschwar,
\textit{Time-analyticity of solutions to the Ricci flow},
Amer. J. Math (to appear), e-Print: {\tt  arXiv:1210.3083v2 [math.DG]}.

\bibitem[Ko4]{KotschwarHOCFU} B.\ Kotschwar
\textit{An energy approach to uniqueness for higher-order geometric flows},
Preprint (2014).


\bibitem[KW]{KotschwarWang} B.\ Kotschwar and L.\ Wang,
\textit{Rigidity of asymptotically conical shrinking gradient Ricci solitons}, 
  J. Diff. Geom. (to appear), e-Print: {\tt arXiv:1307.4746 [math.DG]}.

 \bibitem[Ku1]{Kukavica1} I.\ Kukavica,
 \textit{Backward uniqueness for solutions of linear parabolic equations},
  Proc. Amer. Math. Soc. {\bf 132} (2004), no. 6, 1755--1760.

\bibitem[Ku2]{Kukavica2} I.\ Kukavica, 
  \textit{Log-log convexity and backward uniqueness}, 
 Proc. Amer. Math. Soc. {\bf 135} (2007), no. 8, 2415--2421.
 
%
 \bibitem[LP]{LeesProtter} M.\ Lees and M.\ H.\ Protter,
   \textit{Unique continuation for parabolic differential equations 
     and inequalities},
   Duke Math. J. {\bf 28} (1961), 369--382. 
%
%
%
%

\bibitem[O]{Ogawa} H.\ Ogawa, 
\textit{Lower bounds for solutions of differential inequalities in Hilbert space},
Proc. Amer. Math. Soc. {\bf 16} (1965), 1241--1243.

 
 \bibitem[Po]{Poon} C. C. Poon,
 	\textit{Unique continuation for parabolic equations},
 	Comm. P.\ D.\ E.\ {\bf 21} (1996), no. 3-4, 521--539.
%

\bibitem[Pr]{Protter} M.\ H. Protter,
\textit{Properties of solutions of parabolic equations and inequalities},
  Canad.\ J.\ Math.\ {\bf 28} (1961), 331--345.

 \bibitem[SS]{SautScheurer} J.\ - C. Saut and B.\ Scheurer,
 	\textit{Unique continuation for some evolution equations},
 	J. Diff. Eq. {\bf 66}  (1987),  no. 1, 118--139. 
%

\bibitem[Sha]{Shao1} Y.\ Shao,
\textit {A family of parameter-dependent diffeomorphisms acting on function spaces 
over a Riemannian manifold and applications to geometric flows},
 Nonlin. Diff. Eq. Appl. (2013), 1--41.

 \bibitem[Shi]{Shi} W.\ X.\ Shi,  
   \textit{ Deforming the metric on complete Riemannian manifolds.}  
   J. Diff. Geom.  {\bf 30}  (1989),  no. 1, 223--301.
%
 \bibitem[So]{Sogge} C.\ D.\ Sogge,
 	\textit{A unique continuation theorem for second order parabolic differential operators},  
 	Ark. Mat. {\bf 28}  (1990),  no. 1, 159--182. 
%
%
%
 \bibitem[St1]{StreetsL2}  J.\ D.\ Streets,
 \textit{The gradient flow of $\int_M|\Rm|^2$}, 
  J. Geom. Anal. \textbf{18} (2008), no. 1, 249--271.
%
 \bibitem[St2]{Streets4th} J.\ D.\ Streets,
 \textit{The long time behavior of fourth order curvature flows}.
 Calc. Var. Partial Differential Equations {\bf 46} (2013), no. 1-2, 39--54. 
%
 \bibitem[T]{Tataru}  D. Tataru,
  \textit{Unique continuation problems for partial differential equations}, 
 	\textit{IMA Vol. Math. Appl.} {\bf 137}, (2003), 239--255.
%
 \bibitem[Ta]{Tam} L.\ -F.\ Tam,   
 \textit{Construction of an exhaustion function on complete manifolds.} Preprint.

\bibitem[W1]{WangMCFCone} L.\ Wang, 
\textit{Uniqueness of self-similar shrinkers with asymptotically conical ends},
J. Amer. Math. Soc. 27 (2014), no. 3, 613--638.

\bibitem[W2]{WangMCFCyl} L.\ Wang,
\textit{Uniqueness of self-similar shrinkers with asymptotically cylindrical ends},
J. Reine Angew. Math., (to appear).
 

\bibitem[WY]{WongYu} W.\ W.\ -Y. Wong and P.\ Yu,
 \textit{On strong unique continuation of coupled Einstein metrics},
 Int. Math. Res. Not. (2012), no. 3, 544--560. 
 	
%
%

\end{thebibliography}
\end{document}